\documentclass[11pt]{amsart}

\usepackage[english]{babel}     % include this always
\usepackage[utf8]{inputenc}
\usepackage{amsmath}
\usepackage{amssymb}
\usepackage{mathrsfs}
\usepackage{stmaryrd}
\usepackage{hyperref}
\usepackage{xcolor}
\usepackage{todonotes}
\usepackage{comment}
\usepackage{enumerate}
\usepackage{tikz-cd}

\newtheorem{theorem}{Theorem}[section]
\newtheorem{theoremx}{Theorem}

\newtheorem{proposition}[theorem]{Proposition}
\newtheorem{corollary}[theorem]{Corollary}

\theoremstyle{definition}

\newtheorem{remark}[theorem]{Remark}

\theoremstyle{remark}
\numberwithin{equation}{section}

\let\epsilon\varepsilon

\newcommand{\C}			{{\operatorname{\mathbb{C}}}}
\newcommand{\R}			{{\operatorname{\mathbb{R}}}}

\newcommand{\cL}{\mathcal{L}}

\DeclareMathOperator{\supp}{supp}

\DeclareMathOperator{\id}{id}

\newcommand{\cU}{\mathcal{U}}

\newcommand{\SL}{\textrm{SL}}
\newcommand{\cS}{\mathcal{S}}

\renewcommand{\l}{\lambda}

\title[Multilinear transference of Fourier and Schur multipliers]{Multilinear transference of Fourier and Schur multipliers acting on non-commutative $L_p$-spaces}

\date{\noindent \today.  MSC2010 keywords: 22D25, 43A15,  46L51.  MC, GV, AK are supported by the NWO Vidi grant VI.Vidi.192.018 `Non-commutative harmonic analysis and rigidity of operator algebras'.  }

\author[Caspers]{Martijn Caspers}
\author[Krishnaswamy-Usha]{Amudhan Krishnaswamy-Usha}
\author[Vos]{Gerrit Vos}

\address{TU Delft, EWI/DIAM,
	P.O.Box 5031,
	2600 GA Delft,
	The Netherlands}

\email{M.P.T.Caspers@tudelft.nl}
\email{A.KrishnaswamyUsha@tudelft.nl}
\email{G.M.Vos@tudelft.nl}

\begin{document}

\begin{abstract}
Let $G$ be a locally compact unimodular group, and let $\phi$ be some function of $n$ variables on $G$. To such a $\phi$, one can associate a multilinear Fourier multiplier, which acts on some $n$-fold product of the non-commutative $L_p$-spaces of the group von Neumann algebra. One may also define an associated Schur multiplier, which acts on an $n$-fold product of  Schatten classes $S_p(L_2(G))$. We generalize well-known transference results from the linear case to the multilinear case. In particular, we show that the so-called `multiplicatively bounded $(p_1,\ldots,p_n)$-norm' of a multilinear Schur multiplier is bounded above by the corresponding multiplicatively bounded norm of the Fourier multiplier, with equality whenever the group is amenable. Further, we prove that the bilinear Hilbert transform is not bounded as a vector valued map $L_{p_1}(\mathbb{R}, S_{p_1}) \times L_{p_2}(\mathbb{R}, S_{p_2}) \rightarrow L_{1}(\mathbb{R}, S_{1})$, whenever $p_1$ and $p_2$ are such that $\frac{1}{p_1} + \frac{1}{p_2} = 1$. A similar result holds for  certain Calder\'on-Zygmund type operators. This is in contrast to the non-vector valued Euclidean case.
\end{abstract}

\maketitle

\section{Introduction}

In recent years the analysis of Fourier multipliers on non-commutative $L_p$-spaces has seen a rapid development. In particular several multiplier theorems have been established for the non-commutative $L_p$-spaces of a group von Neumann algebra, see e.g. \cite{JMPGafa}, \cite{MeiRicard}, \cite{CCP}, \cite{MRX},  \cite{PRS}, \cite{Tab}. Here the symbol of the multiplier is a function on a locally compact group and the multiplier acts on the  non-commutative $L_p$-space. In particular, the group plays the role of the frequency side.

In several of these approaches Schur multipliers are used to estimate the bounds of Fourier multipliers and vice versa. For instance, upper bounds on the norms of Fourier multipliers in terms of Schur multipliers play a crucial role in \cite{PRS}. Conversely,  transference from Fourier to Schur multipliers was used by Pisier \cite{Pisier98} to provide examples of bounded multipliers on $L_p$-spaces that are not completely bounded. In \cite{LafforgueDeLaSalle} analogous transference techniques were used to provide examples of non-commutative $L_p$-spaces without the completely bounded approximation property. Further, the use of multilinear Schur multipliers and operator integrals led to several surprising results such as the resolution of Koplienko's conjecture on higher order spectral shift \cite{PSSInventiones}, see also \cite{PSSTAdvances}.

\vspace{0.3cm}

Bo\.{z}ejko and Fendler proved the following in  \cite{BozejkoFendler}. Let $G$ be a locally compact group. Let  $\phi \in C_b(G)$ and set $\widetilde{\phi}(s,t) = \phi(s t^{-1}), s,t \in G$. Then the Schur multiplier $M_{\widetilde{\phi}}: B(L_2(G)) \rightarrow B(L_2(G))$ is bounded if and only if it is completely bounded if and only if the Fourier multiplier $T_{\phi}: \mathcal{L}G \rightarrow \mathcal{L}G$ is completely bounded.

Several papers have treated the extension of the Bo\.{z}ejko-Fendler result to non-commutative $L_p$-spaces. In particular, in \cite{NeuwirthRicard} Neuwirth and Ricard proved for a discrete group $G$ that
\[
\Vert M_{\widetilde{\phi}}: S_p(L_2(G)) \rightarrow S_p(L_2(G)) \Vert_{cb} \leq
\Vert T_\phi: L_p(\mathcal{L}G) \rightarrow L_p(\mathcal{L}G) \Vert_{cb}.
\]
If $G$ is moreover amenable then this is an equality. The same result was then obtained for $G$ a locally compact group in \cite{CaspersDeLaSalle}. An analogous result was obtained for actions and crossed products by Gonz\'alez-Perez \cite{gonzalez2018crossed} and in an ad hoc way in the bilinear discrete setting a similar result was obtained for the discrete Heisenberg group in  \cite[Section 7]{CJKM}.

The purpose of this paper is to prove transference results for Fourier and Schur multipliers in the multilinear setting for arbitrary unimodular locally compact groups. We confine ourselves to the unimodular setting for reasons further discussed in Remark \ref{Rmk=NonUnimodular}.  \\

Now we describe in more detail the contents of the paper. Our first main result (Theorem \ref{thm:transference}) is the following multilinear extension of \cite[Theorem 4.2]{CaspersDeLaSalle}. The definition of $(p_1, \dots, p_n)$-multiplicatively bounded maps is given in  Section \ref{Sect=MultilinearNorms}. If $n=1$ then we are in the linear case and by a well-known theorem of Pisier \cite{Pisier98} a map is `$p$-multiplicatively bounded' if and only if it is completely bounded as a map on the $L_p$-space with the natural operator space structure that was also introduced in  \cite{Pisier98}.

\begin{theoremx}
Let $G$ be a locally compact second countable unimodular group and let $1 \leq p \leq \infty$, $1<p_1, \ldots, p_n \leq \infty$  be such that $p^{-1} =  \sum_{i=1}^n p_i^{-1}$.
Let $\phi \in C_b(G^{\times n})$ and set $\widetilde{\phi} \in C_b(G^{\times n + 1})$ by
\begin{equation} \label{Eqn=phitilde}
    \widetilde{\phi}(s_0, \ldots, s_n) = \phi(s_0 s_1^{-1}, s_1 s_2^{-1}, \ldots, s_{n-1} s_n^{-1}), \qquad s_i \in G.
\end{equation}
If $\phi$ is the symbol of a $(p_1, \ldots, p_n)$-multiplicatively bounded Fourier multiplier $T_\phi$ of $G$, then $\widetilde{\phi}$ is the symbol of a $(p_1, \ldots, p_n)$-multiplicatively bounded Schur multiplier $M_{\widetilde{\phi}}$ of $G$. Moreover,
\[
\begin{split}
& \Vert M_{\widetilde{\phi}}: S_{p_1}(L_2(G)) \times \ldots \times S_{p_n}(L_2(G)) \rightarrow S_{p}(L_2(G)) \Vert_{(p_1,\ldots,p_n)-mb}\\
& \qquad  \leq
\Vert T_{\phi}: L_{p_1}(\mathcal{L}G  ) \times \ldots \times  L_{p_n}(\mathcal{L}G ) \rightarrow L_{p}(\mathcal{L}G ) \Vert_{(p_1,\ldots,p_n)-mb}.
\end{split}
\]
\end{theoremx}

Note that in the linear (unimodular, second countable) case of $n=1$, this result is actually a strengthening of \cite[Theorem 4.2]{CaspersDeLaSalle}. Namely, the symbol here is only assumed to define an $L_p$-Fourier multiplier at a single exponent $p$, whereas \cite[Theorem 4.2]{CaspersDeLaSalle} requires that the multiplier is (completely) bounded for all $1 \leq p \leq \infty$ simultaneously. The current proof takes a different route and uses results that appeared after \cite{CaspersDeLaSalle} in the papers  \cite{CPPR}, \cite{RicardIsrael} and \cite{CJKM}. \\

For transference in the other direction we need our group $G$ to be amenable, just as in \cite{NeuwirthRicard}, \cite{CaspersDeLaSalle}. In fact, amenability is a necessary requirement for our proof strategy, see \cite[Theorem 2.1]{CaspersDeLaSalle}.   The following is our second main result (Corollary \ref{Cor = AmenableTransference}). Note here the strict bounds $1 < p < \infty$, caused by the requirement that the maps $i_q$ from Theorem \ref{thm:amenable-intertwining} are complete isometries.

\begin{theoremx}
Let $G$ be a locally compact unimodular amenable group and let $1 <  p,p_1,\ldots,p_n < \infty$ be such that $p^{-1}=\sum_{i=1}^n p_i^{-1}$. Let $\phi \in C_b(G^{\times n})$ and define $\widetilde{\phi}$ as in \eqref{Eqn=phitilde}.  If $\widetilde{\phi}$ is the symbol of a $(p_1, \ldots, p_n)$-bounded (resp. multiplicatively bounded) Schur multiplier then $\phi$ is the symbol of a $(p_1, \dots, p_n)$-bounded (resp. multiplicatively bounded) Fourier multiplier.  Moreover,
\[
\Vert T_\phi \Vert_{(p_1,\ldots,p_n)} \leq \Vert M_{\widetilde{\phi}} \Vert_{(p_1,\ldots,p_n)}, \qquad
\Vert T_\phi \Vert_{(p_1,\ldots,p_n)-mb} \leq \Vert M_{\widetilde{\phi}} \Vert_{(p_1,\ldots,p_n)-mb},
\]
with equality in the $(p_1,\ldots,p_n)$-$mb$ norm when $G$ is second countable.
\end{theoremx}

The proof is a multilinear version of the ultraproduct techniques from \cite[Theorem 5.2]{CaspersDeLaSalle} and \cite{NeuwirthRicard}.  \\

In a final section, which can mostly be read separately from the rest of the paper, we consider the case of vector valued bilinear Fourier multipliers on $\R$. Lacey and Thiele have shown in \cite{LaceyThiele} that the bilinear Hilbert transform is bounded from $L_{p_1}(\mathbb{R}) \times L_{p_2}(\mathbb{R}) \to L_p(\mathbb{R})$, when  $\frac{2}{3} < p < \infty$ and $\frac{1}{p}=\frac{1}{p_1}+\frac{1}{p_2}$. { The vector valued bilinear Hilbert transform is bounded as a map from 
\[ L_{p_1}(\mathbb{R},S_{q_1})\times L_{p_2}(\mathbb{R},S_{q_2})\to L_p(\mathbb{R},S_q)\] whenever $1<\frac{1}{\max\{q,q'\}} + \frac{1}{\max\{q_1,q_1'\} }+ \frac{1}{\max\{q_2,q_2'\}}$, as shown by Amenta and Uraltsev in \cite{Amenta} and Di Plinio, Li, Martikainen and Vourinen in \cite{DiPlinio22}. In particular, this class does not include H\"older combinations of $q_i$. We show that this result does not extend to the case when $p_i=q_i,p=q=1$, using a transference method   similar to the ones used in earlier sections}. To be precise, we prove the following result (Theorem \ref{Thm=LowerBoundThiele}).

\begin{theoremx}
Let $1  < p_1, p_2 < \infty$ be such that $\frac{1}{p_1} + \frac{1}{p_2} = 1$ and set $h(s,t) = \chi_{\geq 0}(s-t)$.
There exists an absolute constant $C >0$ such that for every $N \in \mathbb{N}_{\geq 1}$ we have
\[
\Vert T_h^{(N)}: L_{p_1}(\mathbb{R}, S_{p_1}^N ) \times  L_{p_2}(\mathbb{R}, S_{p_2}^N ) \rightarrow L_{1}(\mathbb{R}, S_{1}^N ) \Vert > C \log(N).
\]
\end{theoremx}

Additionally, we show a similar result for Calder\'on-Zygmund operators. Here Grafakos and Torres \cite{GrafakosTorres} have shown that for a class of Calder\'on-Zygmund operators we have boundedness $L_1 \times L_1 \rightarrow L_{\frac{1}{2}, \infty}$ in the Euclidean case.  Later, a vector valued extension was obtained in \cite{DiPlinio20}. Here for a class of Calder\'on-Zygmund operators the boundedness of the vector valued map was obtained for $L_{p_1} \times L_{p_2} \rightarrow L_p$ with $1 < p, p_1, p_2 < \infty$ and $\frac{1}{p}=\frac{1}{p_1}+\frac{1}{p_2}$. Theorem \ref{Thm=CZLowerBound} shows that the latter result cannot be extended to the case when $p=1$.

\vspace{0.3cm}

The structure of this paper is as follows. In Section \ref{Sect=Preliminaries} we treat the necessary preliminaries and establish the definitions of multilinear Fourier and Schur multipliers. We also look briefly at transference in the case that $p_i = \infty$ for all $i$. In Section \ref{Sect=Fourier->Schur} we prove the transference from Fourier to Schur multipliers for general $p_i$; the other direction for amenable $G$ is proven in Section \ref{Sect=Schur->Fourier}. Finally, in Section \ref{Sect=HilbertTransform}, we give the counterexamples for the vector valued bilinear Hilbert transform and Calder\'on-Zygmund operators.

\vspace{0.3cm}

\noindent {\it Acknowledgements.} We would like to thank Cl\'ement Coine, { Francesco Di Plinio}, Christian Le Merdy, \'Eric Ricard, Lyudmyla Turowska, and  Gennady Uraltsev for useful discussions, feedback and/or communication about the contents of this paper.

\section{ Preliminaries } \label{Sect=Preliminaries}

\subsection{Notational conventions}
$\mathbb{N}$ denotes the natural numbers starting from 0 and $\mathbb{N}_{\geq 1}$ denotes $\mathbb{N} \backslash \{ 0 \}$.
$M_n := M_n(\mathbb{C})$ denotes the complex $n \times n$ matrices. We denote $B(H)$ for the bounded operators on a Hilbert space $H$.
We denote by $E_{st}, 1 \leq s,t \leq n$ the matrix units of $M_n$. Likewise $E_{st}$, $s,t \in F$ denote the matrix units of $B(\ell_2(F))$ whenever $F$ is a finite set.
We also use $1_F$ to denote the indicator function on the set $F$.

\subsection{Locally compact groups}
Let $G$ be a locally compact group which we assume to be unimodular with Haar measure $\mu_G$, see Remark \ref{Rmk=NonUnimodular}. Integration against the Haar measure is denoted by $\int \cdot \: ds$. For $1 \leq p <\infty$ we let $L_p(G)$ be the $p$-integrable functions with norm determined by $\Vert f \Vert_p^p = \int \vert f(s) \vert^p ds$. $C_c(G)$ denotes the continuous and compactly supported functions on $G$.
$L_1(G)$ is a $\ast$-algebra with multiplication given by convolution $(f\ast g)(t) = \int f(s) g(s^{-1} t) ds$ and involution given by $f^\ast(s) = \overline{ f(s^{-1})}$.
$\lambda$ denotes the left regular representation of $G$ on $L_2(G)$, i.e. $(\lambda_s f)(t) = f(s^{-1}t)$.  $\lambda$ also determines a representation of $L_1(G)$ by the strongly convergent integral $\lambda(f)   = \int f(s) \l_s  ds$. The Fourier algebra \cite{Eymard} is defined as
\begin{equation} \label{Eqn=FourierAlg}
A(G) := L_2(G) \ast L_2(G) = \{ s \mapsto \langle \l_s \xi, \eta \rangle \mid \xi, \eta \in L_2(G) \}.
\end{equation}
  Set the group von Neumann algebra
\[
    \cL G = \{\lambda_s \mid  s \in G\}'' = \{\lambda(f) \mid f \in L_1(G)\}''.
\]
  It comes equipped with a natural weight $\varphi$ called the \emph{Plancherel weight} that is given, for $x \in \cL G$, by
  \[
  \varphi(x^\ast x) = \left\{
  \begin{array}{ll}
   \Vert f \Vert_2^2 & \textrm{ if } \exists f \in L_2(G) \textrm { s.t. } \forall \xi \in C_{c}(G): x \xi = f \ast \xi, \\
   \infty & \textrm{ otherwise}.
  \end{array}
  \right.
  \]
 $\varphi$ is tracial since (in fact if and only if) $G$ is unimodular, i.e. $\varphi(x^\ast x ) = \varphi(x x^\ast)$.

\subsection{Non-commutative $L_p$ spaces}
Let $L_p(\cL G)$ denote the non-commutative $L_p$ space associated with $\cL G$ and the Plancherel weight $\varphi$. Since $\varphi$ is a trace, this space can be viewed as the completion of the set of elements in $\cL G$ with finite $\|.\|_{L_p(\cL G)}$ norm:
\[ L_p(\cL G) = \overline{\left\{ x \in \cL G : \| x \|_{L_p(\cL G)} = \varphi( |x|^p)^{1/p} < \infty \right\}}^{\|.\|_{L_p(\cL G)}}.\]
Let $C_c(G) \star C_c(G)$ denote the span of the set of functions of the form $f_1 \ast f_2, f_i \in C_c(G)$. Then $\lambda(C_c(G)\star C_c(G))$   is dense in $L_p(\cL G)$ for every $1\leq p <\infty$ and is weak* dense in $L_\infty(\cL G)=\cL G$.

\subsection{Multilinear Fourier multipliers}

Let $\phi \in C_b(G^{\times n})$. The Fourier multiplier $T_\phi$ associated with the symbol $\phi$ is the multilinear map defined for $\lambda(f_i), f_i \in C_c(G)\star C_c(G)$ by
\[
T_\phi(\lambda(f_1),\ldots, \lambda(f_n)) = \int_{G^{\times n}} \phi(t_1,\ldots,t_n) f_1(t_1)\ldots f_n(t_n) \lambda(t_1\ldots t_n) dt_1 \ldots dt_n.
\]
Let $1\leq p_1, \ldots, p_n, p < \infty$ with $p^{-1}=\sum_i p_i^{-1}$.   Assume $T_\phi$ maps $\lambda(C_c(G)\star C_c(G)) \times \ldots \times \lambda(C_c(G)\ast C_c(G))$ into $L_p(\cL G)$. Equip the $i$-th copy of $\lambda(C_c(G)\star C_c(G))$ with the $\|.\|_{L_{p_i}(\cL G)}$ topology. If $T_\phi$ is a continuous multilinear map with respect to the above topologies, we extend $T_\phi$ to a map $L_{p_1}(\cL G)\times \ldots \times L_{p_n}(\cL G) \to L_p (\cL G)$. By mild abuse of notation, we also denote this map by $T_\phi$, and call it the $(p_1,\ldots,p_n)$-Fourier multiplier associated with $\phi$. When some or all of the $p_i, p$ are equal to $\infty$, we equip the corresponding copy of $\lambda(C_c(G) \star C_c(G))$ with the norm topology from $C^\ast_\lambda(G)$, and replace $L_{p_i}(\cL G)$ by $C^\ast_\lambda(G)$.

By a closed graph argument, $T_\phi$ is then a bounded multilinear map, and we denote its norm by $\Vert T_\phi \Vert_{(p_1,\ldots,p_n)}$.\\

\subsection{Schatten $p$-operators} \label{Schatten p operators}
For $1\leq p< \infty$  let $S_p(H)$ denote the Schatten $p$-operators on a Hilbert space $H$ consisting of all $x \in B(H)$ such that $\Vert x \Vert_{S_p} := {\rm Tr}(\vert x \vert^p)^{1/p} <\infty$ where ${\rm Tr}$ is the usual trace on $B(H)$. $S_\infty(H)$ denotes the compact operators on $H$.   For $1 \leq p \leq q \leq \infty$, we have the dense inclusions $S_p(H) \subseteq S_q(H)$.

Again, let $G$ be a unimodular locally compact group. For $F \subset G$  a relatively compact set with positive measure let $P_F:L_2(G)\to L_2(F)$ be the orthogonal projection. Then, for $1\leq p \leq \infty$, and $x \in L_{2p}( \cL G)$, $x P_F$ defines an operator in $S_{2p}(L_2(G))$ (see \cite[Proposition 3.1]{CaspersDeLaSalle}). For $x \in L_{p}(\cL G)$ with polar decomposition $x = u \vert x \vert$, we will abusively denote by $P_F x P_F$ the operator $ ( \vert x \vert^{1/2} u^\ast P_F)^\ast     \vert x \vert^{1/2} P_F $, which lies in $S_p(L_2(G))$ whenever $x \in L_p(\cL G)$. We will additionally use the fact that the map 
\[
x \mapsto  \mu_G (F )^{-1/p} P_F x P_F
\]
defines a contraction from $L_p(\cL G)$ to $S_p(L_2(G))$ \cite[Theorem 5.1]{CaspersDeLaSalle}.

Let $E$ be an operator space. For $N \in \mathbb{N}_{\geq 1}$, and $1\leq p \leq \infty$, $S_p^N[E]$ will denote the space $M_N(E)$ equipped with the `operator valued Schatten $p$-norm'. When $E=\C$, this is the Schatten $p$-class associated with a Hilbert space of dimension $n$, or equivalently, the non-commutative $L_p$-space associated with $M_N$ equipped with the normalized trace, and is denoted by just $S_p^N$. When $p=\infty$, the norm on $S_p[E]$ is the operator space norm on $M_N(E)$; for $p=1$, this is the projective operator space norm on $S_1^N\otimes E$. The rest are constructed via interpolation. The particulars will not be used in what follows, we refer to \cite{Pisier98} for the details. If $E$ is the non-commutative $L_p$ space associated with some tracial von Neumann algebra $\mathcal{M}$, $S_p^N[E]$ can be identified with the non-commutative $L_p$ space corresponding to $M_N \otimes \mathcal{M}$.

\subsection{Multilinear Schur multipliers}
Let $X$ be some measure space.  We identify $S_2(L_2(X))$ linearly and isometrically with the integral operators given by kernels in $L_2(X\times X)$. This way $A \in L_2(X\times X)$ corresponds to $(A \xi)(t) = \int A(t,s) \xi(s) ds$. Throughout what follows, we will make no distinction between a Hilbert-Schmidt operator on $L_2(X)$ and its kernel in $L_2(X \times X)$.\\

For $\phi \in L_\infty(X^{\times n+1})$ the associated Schur multiplier is the multilinear map $S_2(L_2(X))\times\ldots \times S_2(L_2(X)) \to S_2(L_2(X))$ determined by
\[
M_\phi( A_1,\ldots, A_n)(t_0,t_n) = \int_{X^{\times n-1}} \phi(t_0,\ldots,t_n) A_1(t_0,t_1)A_2(t_1,t_2)\ldots A_n(t_{n-1},t_n) dt_1 \ldots dt_{n-1}.
\]
That $M_\phi$ indeed takes values in $S_2(L_2(X))$, or rather $L_2(X \times X)$, is an easy application of the Cauchy-Schwarz inequality, as we show here in the case of $n=2$:
\[
\begin{split}
    \iint_{X^2} \left| \int_X \phi(r,s,t) A(r,s) B(s,t) ds\right|^2 drdt
    \leq & \|\phi\|_{\infty}^2 \int_X \int_X \left( \int_X |A(r,s)|^2 ds\right) dr \left( \int_X |B(s,t)|^2 ds\right) dt \\
    =& \|\phi\|_\infty^2 \|A\|_2^2 \|B\|_2^2.
\end{split}
\]
The case of higher order $n$ is similar to \cite[Lemma 2.1]{PSSTAdvances}.

Let $1\leq p,p_1,\ldots,p_n \leq \infty$, with $p^{-1}=\sum_{i=1}^n p_i^{-1}$. Consider the restriction of $M_\phi$ where its $i$-th input is restricted to $S_2(L_2(X)) \cap S_{p_i}(L_2(X))$. If the resulting restriction takes values in $S_p(L_2(X))$ and has a bounded extension to $S_{p_1}(L_2(X)) \times \ldots \times S_{p_n}(L_2(X))$, its extension, also denoted by $M_\phi$, is called the $(p_1, \dots, p_n)$-Schur multiplier.

\subsection{Norms for multilinear maps} \label{Sect=MultilinearNorms}

If $E_1,\ldots, E_n, E$ are Banach spaces and $T: E_1 \times \ldots \times E_n \to E$ is a multilinear map, we recall that $\| T \|$ is the quantity $\sup_{x_i \in E_i, \| x_i \| =1} \| T(x_1,\ldots,x_n)\|$.

If $E_1, \ldots, E_n, E$ are operator spaces, $T$ is a multilinear map, and $N \in \mathbb{N}_{\geq 1}$, the multiplicative amplification of $T$ refers to the map

\begin{equation}\label{Eqn=NAmplify}
T^{(N)}: M_N(E_1) \times \ldots \times M_N(E_n) \to M_N(E)
\end{equation}
which sends $x_i = \alpha_i \otimes v_i$ where $\alpha_i \in M_N, v_i \in E_i$, to $(\alpha_1 \ldots \alpha_n) \otimes T(v_1,\ldots,v_n)$. $T^{(N)}$ is viewed as a multilinear map on the space $M_N(E_i)$, which is equipped with the matrix norm from the operator space structure of $E_i$. We say $T$ is multiplicatively bounded if $\|T\|_{mb}=\sup_N \| T^{(N)}\| < \infty$.\\

Let $1\leq p, p_1, \ldots p_n \leq \infty$ with $p^{-1}=\sum_{i=1}^n p_i^{-1}$. Generalizing a definition by \cite{Xu}, we will say a multilinear map $T: E_1 \times \ldots  \times E_n \to E $ is $(p_1,\ldots,p_n)$-multiplicatively bounded if the multiplicative extensions
\[
 T^{(N) } :=  T^{(N)}_{p_1,\ldots,p_n}: S^{N}_{p_1}[E_1]\times \ldots \times S^{N}_{p_n}[E_n] \to S^{N}_p[E]
\]
have uniformly bounded norms. The $(p_1,\ldots,p_n)$-multiplicatively bounded norm of $T$ is then $\sup_N \Vert T^{(N)} \Vert$. It is denoted by $\|T\|_{(p_1, \dots, p_n)-mb}$.

\begin{remark} It is unclear if this definition of $(p_1, \ldots, p_n)$-multiplicative boundedness corresponds to complete boundedness of some linear map on some appropriate tensor product of the $E_i$'s. In the special case the range space is $\C$ and $n=2$ such a tensor product has been constructed in  \cite[Remark 2.7]{Xu}.   However, this tensor norm does not seem to admit a natural operator space structure, nor does it seem to work in the multilinear case.
\end{remark}

The norms of multilinear Schur multipliers are determined by the restriction of the symbol to finite sets. This is the multilinear version of \cite[Theorem 1.19]{LafforgueDeLaSalle} and \cite[Theorem 3.1]{CaspersDeLaSalle}.

\begin{theorem} \label{Thm=Finite truncation}
Let $\mu$ be a Radon measure on a locally compact space $X$, and $\phi: X^{n+1} \to \C$ a continuous function. Let $K>0$. The following are equivalent for $1 \leq p_1, \dots, p_n, p \leq \infty$:
\begin{enumerate}[(i)]
    \item $\phi$ defines a bounded Schur multiplier $S_{p_1}(L_2(X)) \times \ldots \times S_{p_n}(L_2(X)) \to S_p(L_2(X))$ with norm less than $K$.
    \item For every $\sigma$-finite measurable subset $X_0$ in $X$, $\phi$ restricts to a bounded Schur multiplier $S_{p_1}(L_2(X_0)) \times \ldots \times S_{p_n}(L_2(X_0)) \to S_p(L_2(X_0))$ with norm less than $K$.
    \item For any finite subset $F = \{s_1, \dots, s_N\} \subset X$ belonging to the support of $\mu$, the symbol $\phi|_{F^{\times (n+1)}}$ defines a bounded Schur multiplier $S_{p_1}(\ell_2(F)) \times \ldots \times S_{p_2}(\ell_2(F)) \to S_p(\ell_2(F))$ with norm less than $K$.
\end{enumerate}
The same equivalence is true for the $(p_1, \ldots, p_n)-mb$ norms.
\end{theorem}

\begin{proof}
$(i) \Rightarrow (ii)$ is trivial. The implication $(ii) \Rightarrow (i)$ remains exactly the same as in \cite[Theorem 3.1]{CaspersDeLaSalle} except for the fact that we have to take $x_i \in S_{p_i}(L_2(X))$ and take into account the support projections of $x_1, x_1^*, \ldots, x_n, x_n^*$ when choosing $X_0$. The equivalence $(ii) \Leftrightarrow (iii)$ is mutatis mutandis the same as in \cite[Theorem 1.19]{LafforgueDeLaSalle}. For the $(p_1, \ldots, p_n)-mb$ norms, we apply the theorem on the space $X_N = X \times \{1, \dots, N\}$ and function $\phi_N((s_0, i_0), \dots, (s_n, i_n)) = \phi(s_0, \dots, s_n)$ and use the isometric identifications
\[ S_q(L_2(X_N)) \cong S_q(L_2(X) \otimes \C^N) \cong S^N_q(S_q(L_2(X))).
\]
and the fact that under these identifications we have
\[
    M_\phi^{(N)}(x_1, \dots, x_n) = M_{\phi_N}(x_1, \dots, x_n), \qquad x_i \in S_{p_i}^N(S_{p_i}(L_2(X))).
\]
\end{proof}

\subsection{Transference for $p_i = \infty$}\label{Sect=PisInfinity}

Let $G$ be a locally compact group which in this subsection is not required to be unimodular. The following   Proposition \ref{Prop=Multilinear Bozejko-Fendler} (based on \cite[Theorem 4.1]{EffrosRuanAMS})  was proved in \cite[Theorem 5.5]{TodorovTurowska}. This is a multilinear version of the Bo\.{z}ejko-Fendler result \cite{BozejkoFendler} and it yields a transference result between Fourier and Schur multipliers for the case $p_1 = \ldots = p_n = \infty$.  We give a proof of the if direction that is slightly different from  \cite{TodorovTurowska} by using the transference techniques from Theorem \ref{thm:transference}, which simplify in the current setup.

\begin{proposition} \label{Prop=Multilinear Bozejko-Fendler}
 For $\phi \in C_b(G^{\times n})$ set  \[
 \widetilde{\phi}(s_0, \ldots, s_n) = \phi(s_0 s_1^{-1}, s_1 s_2^{-2}, \ldots, s_{n-1} s_n^{-1}), \qquad s_i \in G.
 \]
Then $M_{\widetilde{\phi}}$ is multiplicatively bounded on $S_\infty(L_2(G))^{\times n} \to S_\infty(L_2(G))$  iff $T_\phi$ defines a multiplicatively bounded multilinear map on $\cL G^{\times n}
\to \cL G$. In this case, we have
\[ \|T_\phi\|_{mb} = \|M_{\widetilde{\phi}}\|_{mb} \]
\end{proposition}

\begin{proof}[Proof of the "if" direction]

Assume that $T_\phi$ is multiplicatively bounded. Let $F \subseteq G$ finite with $|F| = N$. Let $p_s \in B(\ell_2(F))$ be the projection on the one dimensional space spanned by the delta function $\delta_s$. Let $\widetilde{\phi}_F := \widetilde{\phi}|_{F^{\times n+1}}$. By Theorem \ref{Thm=Finite truncation} (using that $\phi$ is continuous), it suffices to prove that
\[
    M_{\widetilde{\phi}_F}: B(\ell_2(F))^{\times n} \to B(\ell_2(F))
\]
and its matrix amplifications are bounded by $\|T_\phi\|_{mb}$. Define the unitary $U = \sum_{s \in F} p_s \otimes \lambda_s \in B(\ell_2(F)) \otimes \cL G$ and the isometry
\[ \pi: B(\ell_2(F)) \to B(\ell_2(F)) \otimes \cL G, \quad \pi(x) = U(x \otimes \id)U^*. \]
Note that $\pi$ satisfies $\pi(E_{st}) = E_{st} \otimes \lambda_{st^{-1}}$. For $s_0, \dots, s_n \in F$:
\[
\begin{split}
\pi(M_{\widetilde{\phi}_F}(E_{s_0 s_1}, E_{s_1 s_2}, \ldots, E_{s_{n-1} s_n})) &= \pi( \widetilde{\phi}(s_0, \ldots, s_n)E_{s_0 s_n}) \\
&= \phi(s_0 s_1^{-1}, \ldots, s_{n-1} s_n^{-1}) E_{s_0 s_n} \otimes \lambda_{s_0 s_n^{-1}},
\end{split}
\]
while
\[
\begin{split}
    T_{\phi}^{(N)}(\pi(E_{s_0 s_1}), \ldots, \pi(E_{s_{n-1} s_n})) &= T_\phi^{(N)}(E_{s_0s_1} \otimes \lambda_{s_0 s_1}^{-1}, \ldots, E_{s_{n-1} s_n} \otimes \lambda_{s_{n-1} s_n^{-1}})\\
    &= E_{s_0 s_n} \otimes T_\phi(\lambda_{s_0 s_1^{-1}}, \ldots, \lambda_{s_{n-1} s_n^{-1}}) \\
    &= \phi(s_0 s_1^{-1}, \ldots, s_{n-1} s_n^{-1}) E_{s_0 s_n} \otimes \lambda_{s_0 s_n^{-1}}.
\end{split}
\]
It follows that $T_{\phi}^{(N)} \circ \pi^{\times n} = \pi \circ M_{\widetilde{\phi}_F}$.
 This implies that
\[
    \| M_{\widetilde{\phi}_F}\| = \| \pi \circ M_{\widetilde{\phi}_F}\|  = \|T_{\phi}^{(N)} \circ \pi\| \leq \|T_{\phi}^{(N)}\| \leq \|T_{\phi}\|_{mb}.
\]
By taking matrix amplifications, we prove similarly that $\| M_{\widetilde{\phi}_F}\|_{mb} \leq \|T_{\phi}\|_{mb}$.
\end{proof}

\begin{remark} \label{Rmk=automb}
 A multilinear map on the product of some operator spaces is multiplicatively bounded iff its linearization is completely bounded as a map on the corresponding Haagerup tensor product. However, as \cite[Lemma 3.3]{JuschenkoTodorovTurowska} shows, for Schur multipliers $M_{\widetilde{\phi}}$ on $S_\infty(L_2(G))^{\times n}$, just boundedness on the Haagerup tensor product is sufficient to guarantee that $M_{\widetilde{\phi}}$ is multiplicatively bounded. Note that even in the linear case, when $p<\infty$, it is unknown whether a bounded Schur multiplier on $S_p(L_2(\R))$ is necessarily completely bounded unless $\phi$ has continuous symbol (we refer to \cite[Conjecture 8.1.12]{Pisier98},  \cite[Theorem 1.19]{LafforgueDeLaSalle}, \cite{CaspersWildschut}).
\end{remark}

% \begin{remark}
% A closer look at the above proof also reveals that it is sufficient to assume $T_\phi$ defines a multiplicatively bounded map on $C_r^\ast(G)^{\times n} \to C_r^\ast(G)$. This guarantees that $T_\phi$ is in fact a multiplicatively bounded Fourier multiplier on $\cL G$. In other words, we get weak* continuity for free.
% \end{remark}

\section{Transference from Fourier to Schur multipliers} \label{Sect=Fourier->Schur}

Let $G$ be a locally compact group, which is again assumed to be unimodular. We will prove that symbols of $(p_1, \dots, p_n)$-Fourier multipliers are also symbols of $(p_1, \dots, p_n)$-Schur multipliers using a multilinear transference method. This yields a multilinear version of \cite[Theorem 4.2]{CaspersDeLaSalle}. We stipulate that in the proofs of  \cite[Section 4]{CaspersDeLaSalle} the transference is carried out for rational exponents $p$.  In order to treat the general multilinear case we present an alternative proof that transfers multipliers directly for every real exponent $p \in [1, \infty)$. In fact, this gives an improvement of \cite[Theorem 4.2]{CaspersDeLaSalle} which is stated under the stronger assumption that the multiplier acts boundedly on the Fourier algebra (equivalently is a $p$-multiplier at $p=1$).
The fundamental difference in the proof is that we base ourselves on the methods from \cite[Claim B, p. 24]{CPPR} and \cite[Lemma 4.6]{CJKM}.

As before, for   $\phi \in C_b(G^{\times n})$ we set $\widetilde{\phi} \in C_b(G^{\times n+1})$ by
\[
\widetilde{\phi}(s_0, \ldots, s_n) = \phi(s_0 s_1^{-1}, s_1 s_2^{-1}, \ldots, s_{n-1} s_n^{-1}), \qquad s_i \in G.
\]

\begin{theorem}\label{thm:transference}
Let $G$ be a locally compact second countable unimodular group and let $1 \leq p \leq \infty, 1 <  p_1, \ldots, p_n \leq \infty$  be such that $p^{-1} =  \sum_{i=1}^n p_i^{-1}$.
Let $\phi \in C_b(G^{\times n})$ be the symbol of a $(p_1, \ldots, p_n)$-multiplicatively bounded Fourier multiplier of $G$. Then $\widetilde{\phi}$ is the symbol of a $(p_1, \ldots, p_n)$-multiplicatively bounded Schur multiplier of $G$. Moreover,
\[
\begin{split}
& \Vert M_{\widetilde{\phi}}: S_{p_1}(L_2(G)) \times \ldots \times S_{p_n}(L_2(G)) \rightarrow S_{p}(L_2(G)) \Vert_{(p_1,\ldots,p_n)-mb}\\
& \qquad  \leq
\Vert T_{\phi}: L_{p_1}(\mathcal{L}G ) \times \ldots \times  L_{p_n}(\mathcal{L}G ) \rightarrow L_{p}(\mathcal{L}G ) \Vert_{(p_1,\ldots,p_n)-mb}.
\end{split}
\]
\end{theorem}
\begin{proof}

Let $F \subseteq G$ be finite. Consider the Hilbert space $\ell_2(F)$ and for $s \in F$ let $p_s$ be the orthogonal projection onto the one dimensional space spanned by the delta function $\delta_s$. By Theorem \ref{Thm=Finite truncation} it suffices to show that the norm of
\[
M_{\widetilde{\phi}}: \cS_{p_1}(\ell_2(F)) \times \ldots \times  \cS_{p_n}(\ell_2(F)) \rightarrow  \cS_{p}(\ell_2(F)).
\]
and its matrix amplifications are bounded by $\Vert   T_\phi \Vert_{(p_1,\ldots,p_n)-mb}$.

The proof requires the introduction of coordinate-wise convolutions as follows. Fix functions $\varphi_k \in A(G)$  such that $\varphi_k \geq 0$, $\Vert \varphi_k \Vert_{L_1(G)} = 1$ and such that the support of $\varphi_k$ shrinks to the identity of $G$; from \eqref{Eqn=FourierAlg} it is clear that such functions exist. Then for any function $\phi \in C_b( G^{\times n} )$ we set
\begin{equation}\label{Eqn=Convolution}
\begin{split}
& \phi_k(s_1, \ldots, s_n) \\
 := & \int_{G^{\times n}}  \phi(t_1^{-1} s_1 t_2, t_2^{-1} s_2 t_3, \ldots  ,  t_{n-2}^{-1} s_{n-2} t_{n-1}   ,  t_{n-1}^{-1} s_{n-1},   s_n  t_{n}^{-1} ) \left( \prod_{j=1}^{n} \varphi_k(t_j) \right)   dt_1 \ldots dt_n 
%    = & \int_{G^{\times n}}  \phi(t_1^{-1}  t_2, t_2^{-1}   t_3, \ldots,  t_{n-2}^{-1}t_{n-1},       t_{n-1}^{-1},  t_n^{-1} )
  %  \left( \prod_{j=1}^{n-1} \varphi_k(s_j \ldots s_{n-1} t_j) \right) \varphi_k(t_n s_n)  dt_1 \ldots dt_n.
\end{split}
\end{equation}
For the particular case $n = 1$ this expression becomes by definition
\[
\begin{split}
\phi_k(s)
 := & \int_{G}  \phi(  s  t^{-1} )  \varphi_k(t)     dt
    =  \int_{G }  \phi(  t^{-1} )
  \varphi_k(t s)  dt.
\end{split}
\]

Let $(U_\alpha)_\alpha$ be a symmetric neighbourhood basis of the identity of $G$ consisting of relatively compact sets.
  Set
\[
k_\alpha = \vert U_\alpha \vert^{- \frac{1}{2}} \lambda(1_{U_{\alpha}}),
\]
with polar decomposition $k_\alpha = u_\alpha h_\alpha$.
Then $k_\alpha$ is an element in $L_2(\cL G)$ with $\Vert k_\alpha \Vert_2 = 1$. Consequently $h_\alpha^{2/q}$ is   in $L_q(\cL G)$ for $1 \leq q < \infty$ with $\Vert h_\alpha^{2/q} \Vert = 1$. In case $q = \infty$ by mild abuse of notation we set $h_\alpha^{2/q} =1$.  Set the unitary
\[
U = \sum_{s \in F}  p_s \otimes \lambda_s \in B(\ell_2(F)) \otimes \cL G.
\]
Now let $a_i \in S_{p_i}(\ell_2(F))$.
 Since $\phi_k$ converges to $\phi$ pointwise, we have
 \[
 M_{\widetilde{\phi}}(a_1, \ldots, a_n)=\lim_k M_{\widetilde{\phi_k}}(a_1, \ldots, a_n).
 \]
  So with $N = \vert F \vert$,
\begin{equation}\label{Eqn=TransferenceEstimate}
\begin{split}
 \Vert M_{\widetilde{\phi}}(a_1, \ldots, a_n ) \Vert_{S_p^N}
= &  \lim_k \limsup_\alpha \Vert M_{\widetilde{\phi_k}}(a_1, \ldots, a_n ) { \otimes h_\alpha^{\frac2p}} \Vert_{S_p^N{ \otimes L_p(\cL G)}} \\
= & \lim_k \limsup_\alpha \Vert U   ( M_{\widetilde{\phi_k}}(a_1, \ldots, a_n ) \otimes h_\alpha^{\frac{2}{p}} ) U^\ast \Vert_{S_p^N \otimes L_p(\cL G)} \\
\leq & \; A+ B,
\end{split}
\end{equation}
where

\begin{equation}
\begin{split}
&A =  \limsup_k   \limsup_\alpha \Vert  T_{\phi_k}^{ (N) }(  U   ( a_1 \otimes h_\alpha^{\frac{2}{p_1}}  ) U^\ast, \ldots,   U  ( a_n \otimes h_\alpha^{\frac{2}{p_n}}  ) U^\ast   )    \Vert_{S_p^N \otimes L_p(\cL G)}, \\
&B =\limsup_k \limsup_\alpha  \Vert  T_{\phi_k}^{ (N) }(  U   ( a_1 \otimes h_\alpha^{\frac{2}{p_1}}  ) U^\ast, \ldots,    U   ( a_n \otimes h_\alpha^{\frac{2}{p_n}}  ) U^\ast   ) -  \\
&  \hspace{6cm} U   ( M_{\widetilde{\phi_k}}(a_1, \ldots,  a_n) \otimes h_\alpha^{\frac{2}{p}} ) U^\ast    \Vert_{S_p^N \otimes L_p(\cL G)}.
\end{split}
\end{equation}

Below we prove that $B=0$. Therefore,
\begin{equation}\label{Eqn=CrucialEstimate}
\begin{split}
\Vert M_{\widetilde{\phi}} & (a_1, \ldots, a_n ) \Vert_{S_p^N}  \leq  \limsup_k \limsup_\alpha  \Vert  T_{\phi_k}^{(N)}(  U   ( a_1 \otimes h_\alpha^{\frac{2}{p_1}}  ) U^\ast, \ldots,    U   ( a_n \otimes h_\alpha^{\frac{2}{p_n}}  ) U^\ast   )    \Vert_{S_p^N \otimes L_p(\cL G)} \\
\leq &  \limsup_k  \limsup_\alpha  \Vert T_{\phi_k}^{(N)} \Vert \Vert  U   ( a_1 \otimes h_\alpha^{\frac{2}{p_1}}  ) U^\ast \Vert_{S_{p_1}^N \otimes L_{p_1}(\cL G)} \ldots
   \Vert  U   ( a_n \otimes h_\alpha^{\frac{2}{p_n}}  ) U^\ast \Vert_{S_{p_n}^N \otimes L_{p_n}(\cL G)}\\
=&  \limsup_k    \Vert T_{\phi_k}^{(N)} \Vert \Vert a_1 \Vert_{S_{p_1}^N} \ldots \Vert a_n \Vert_{S_{p_n}^N},
\end{split}
\end{equation}
where the norm of $T_{\phi_k}^{(N)}$ is understood as in \eqref{Eqn=NAmplify} for the map $T_{\phi_k}: L_{p_1}(\cL G) \times \ldots \times  L_{p_n}(\cL G) \rightarrow L_{p}(\cL G)$.  By \cite[Lemma 4.3]{CJKM} and the fact that $\Vert \varphi_k \Vert_{L_1(G)} = 1$ it follows then that $\Vert T_{\phi_k}^{(N)} \Vert \leq \Vert T_{\phi}^{(N)}\Vert$. Hence 
\begin{equation}\label{Eqn=ConvApprox}
\Vert M_{\widetilde{\phi}}(a_1, \ldots, a_n ) \Vert_{S_p^N}  \leq   \Vert T_{\phi}^{(N)} \Vert \Vert a_1 \Vert_{S_{p_1}^N} \ldots \Vert a_n \Vert_{S_{p_n}^N}.
\end{equation}
This finishes the proof. The multiplicatively bounded case follows by taking matrix amplifications.

\vspace{0.3cm}

Now let us prove that the last term in \eqref{Eqn=TransferenceEstimate} goes to 0. By the triangle inequality it suffices to prove that the limits of the following terms are 0. For $r_0, \ldots, r_n \in F$  with matrix units $E_{r_i, r_{i+1}}$,
\begin{equation}\label{Eqn=TransferenceFinal}
\begin{split}
&    T_{\phi_k}^{ (N) }(  U   ( E_{r_0,r_1} \otimes h_\alpha^{\frac{2}{p_1}}  ) U^\ast, \ldots,   U   ( E_{r_{n-1},r_n} \otimes h_\alpha^{\frac{2}{p_n}}  ) U^\ast   ) -  U   ( M_{\widetilde{\phi_k}}(E_{r_0,r_1}, \ldots,  E_{r_{n-1}, r_{n}}) \otimes h_\alpha^{\frac{2}{p}} ) U^\ast  \\
& \quad =      E_{r_0, r_n} \otimes \left(  T_{\phi_k}( \lambda_{r_0} h_\alpha^{\frac{2}{p_1}}  \lambda_{r_1}^\ast, \ldots, \lambda_{r_{n-1}}  h_\alpha^{\frac{2}{p_n}} \lambda_{r_n}^\ast ) - \phi_k(r_0 r_1^{-1}, \ldots,  r_{n-1} r_n^{-1})  \lambda_{r_0} h_\alpha^{\frac{2}{p}} \lambda_{r_n}^\ast\right).
\end{split}
\end{equation}
Applying the transformation formula \cite[Lemma 4.3]{CJKM} to the $T_{\phi_k}$ term,
\[
T_{\phi_k}( \lambda_{r_0} h_\alpha^{\frac{2}{p_1}}  \lambda_{r_1}^\ast, \ldots, \lambda_{r_{n-1}}  h_\alpha^{\frac{2}{p_n}} \lambda_{r_n}^\ast) = \lambda_{r_0} T_{\phi_k(r_0 \: \cdot \: r_1^{-1}, \ldots,  r_{n-1}   \: \cdot \: r_{n}^{-1} )  }( h_\alpha^{\frac{2}{p_1}}, \ldots,  h_\alpha^{\frac{2}{p_n}}  ) \lambda_{r_n}^\ast. 
\]
Taking the norm of the expression in equation \eqref{Eqn=TransferenceFinal},

\begin{equation}\label{Eqn=TransferenceFinal2}
\begin{split}
\Vert \eqref{Eqn=TransferenceFinal} \Vert_{S_p^N \otimes L_p(\cL G)}
= &  \Vert T_{\phi_k(r_0 \: \cdot \: r_1^{-1}, \ldots,  r_{n-1}   \: \cdot \: r_{n}^{-1} )  }( h_\alpha^{\frac{2}{p_1}}, \ldots,  h_\alpha^{\frac{2}{p_n}}  ) - \phi_k(r_0 r_1^{-1}, \ldots,  r_{n-1} r_n^{-1})  h_\alpha^{\frac{2}{p}} \Vert_{L_p(\cL G)}.
\end{split}
\end{equation}

  We now claim that $\lim_k \limsup_\alpha$ of this expression yields 0, by almost identical arguments as those used in \cite[Lemma 4.6]{CJKM}. Since we have a couple of differences, namely that we have a translated function $\phi_k(r_0 \: \cdot \: r_1^{-1}, \ldots,  r_{n-1}   \: \cdot \: r_{n}^{-1} )$ and we do not use the SAIN condition (see \cite[Definition 3.1]{CPPR}), we spell out some of the details here.

Let $\zeta:G \to \R_{\geq 0}$ be a continuous compactly supported positive definite function in $A(G)$ with $\zeta(e)=1$, so that $T_\zeta$ is contractive. For $1\leq j\leq n$, let $\zeta_j(s)=\zeta(r_{j-1}^{-1} s r_j)$ and let $\phi(\zeta_1,\ldots,\zeta_n)(s_1,\ldots,s_n)= \phi(s_1,\ldots,s_n)\zeta_1(s_1)\ldots\zeta_n(s_n)$. Then
\[
\begin{split}
 &\Vert T_{\phi_k(r_0 \: \cdot \: r_1^{-1}, \ldots,  r_{n-1}   \: \cdot \: r_{n}^{-1} )  }( h_\alpha^{\frac{2}{p_1}}, \ldots,  h_\alpha^{\frac{2}{p_n}}  ) - \phi_k(r_0 r_1^{-1}, \ldots,  r_{n-1} r_n^{-1})  h_\alpha^{\frac{2}{p}} \Vert_{L_p(\cL G)}   \\
& \leq \Vert \left( (\phi(\zeta_1,\ldots,\zeta_n))_k(r_0 r_1^{-1}, \ldots,  r_{n-1} r_{n}^{-1}) -  \phi_k(r_0 r_1^{-1}, \ldots,  r_{n-1} r_n^{-1}) \right)  h_\alpha^{\frac{2}{p}} \Vert_{L_p(\cL G)} \\
  &+  \Vert T_{(\phi(\zeta_1,\ldots,\zeta_n))_k(r_0 \: \cdot \: r_1^{-1}, \ldots,  r_{n-1}   \: \cdot \: r_{n}^{-1} )  }( h_\alpha^{\frac{2}{p_1}}, \ldots,  h_\alpha^{\frac{2}{p_n}}  ) - (\phi(\zeta_1,\ldots,\zeta_n))_k(r_0 r_1^{-1}, \ldots,  r_{n-1} r_n^{-1})  h_\alpha^{\frac{2}{p}} \Vert_{L_p(\cL G)} \\
  &+  \Vert T_{(\phi(\zeta_1,\ldots,\zeta_n))_k(r_0 \: \cdot \: r_1^{-1}, \ldots,  r_{n-1}   \: \cdot \: r_{n}^{-1} )  }( h_\alpha^{\frac{2}{p_1}}, \ldots,  h_\alpha^{\frac{2}{p_n}}  ) - T_{\phi_k(r_0 \: \cdot \: r_1^{-1}, \ldots,  r_{n-1}   \: \cdot \: r_{n}^{-1} )  }( h_\alpha^{\frac{2}{p_1}}, \ldots,  h_\alpha^{\frac{2}{p_n}}  ) \Vert_{L_p(\cL G)} \\
  &= A_{k,\alpha} + B_{k,\alpha} + C_{k,\alpha}.
\end{split}
\]
The terms  $A_{k,\alpha}$, $B_{k,\alpha}$ and $C_{k,\alpha}$ should be compared to the terms occurring in \cite[Eqn. (4.7)]{CJKM}.
Since $\varphi_k$ have support shrinking to $e$, and $\zeta_j(r_{j-1}r_j^{-1})=1$, it follows that $\lim_k A_{k,\alpha} = 0$.

For fixed $t_1,\ldots, t_n \in G$, define
\[ C_{\alpha}(t_1,\ldots,t_n):=\Vert T_{\phi-\phi(\zeta_1,\ldots,\zeta_n)(t_1^{-1}r_0 \: \cdot \: r_1^{-1}t_2, \ldots, t_{n-1}^{-1} r_{n-2} \: \cdot \: r_{n-1}^{-1}, r_{n-1} \: \cdot \: r_n^{-1} t_n^{-1})}(h_\alpha^{\frac{2}{p_1}},\ldots,h_\alpha^{\frac{2}{p_n}}) \Vert_{L_p(\cL G)}
\]
and for $1\leq j \leq n-2$,
\[
\begin{split}
y_{j, \alpha} &= \lambda_{t_j^{-1} r_{j-1}} h_\alpha^{\frac{2}{p_j}} \lambda_{r_j^{-1}t_{j+1}},  \\
y_{n-1, \alpha} &= \lambda_{t_{n-1}^{-1} r_{n-2}} h_\alpha^{\frac{2}{p_{n-1}}} \lambda_{r_{n-1}^{-1}}, \\
y_{n, \alpha} &= \lambda_{r_{n-1}} h_\alpha^{\frac{2}{p_n}} \lambda_{r_n^{-1} t_n^{-1}}.
\end{split}
\]

By the same arguments as in  \cite[Lemma 4.6, after Eqn. (4.8)]{CJKM}, we get for $1 \leq j \leq n-2$:
\begin{equation}\label{Eqn=Limits}
\begin{split}
 \lim_{\alpha }  \Vert (T_{\zeta_j }- \id)( y_{j, \alpha} ) \Vert_{L_{p_j}(\mathcal{L}G )} = &
 \vert  \zeta( r_{j-1}^{-1} t_j^{-1} r_{j-1} r_j^{-1} t_{j+1} t_j    )  - 1 \vert, \\
  \lim_{\alpha }  \Vert (T_{\zeta_{n-1} }- \id)( y_{n-1, \alpha} ) \Vert_{L_{p_{n-1}}(\mathcal{L}G )} = &
 \vert  \zeta( r_{n-2}^{-1} t_{n-1}^{-1}   r_{n-2}  )  - 1 \vert, \\
  \lim_{\alpha }  \Vert (T_{\zeta_n }- \id)( y_{n, \alpha} ) \Vert_{L_{p_n}(\mathcal{L}G )} = &
 \vert  \zeta(   r_n^{-1}  t_{n}^{-1} r_n    )  - 1 \vert.
 \end{split}
 \end{equation}
 Crucially, here we require this argument from \cite[Lemma 4.6, after Eqn. (4.8)]{CJKM} only for $x_j=1$. Hence we do not require the use of \cite[Lemma 3.15]{CJKM}, which uses the SAIN condition. Further note that in this case, the above equalities are trivially true when $p_j=\infty$ (as $h_\alpha^{2/p_j} = 1$ in that case), so we do not require the use of \cite[Proposition 3.9]{CJKM}, which holds only for $1<p_j<\infty$.
 Integrating $C_{\alpha}(t_1,\ldots,t_n)$ against $\prod_{j=1}^n \varphi_k (t_j)$, we can show that $\lim_k \limsup_\alpha C_{k,\alpha} = 0$, just as in \cite[Lemma 4.6]{CJKM}.

 To show that $\lim_k \lim_\alpha B_{k,\alpha}=0$, we only note the modifications from \cite[Lemma 4.6]{CJKM}. As before,  $x_j=1$ in the proof of \cite[Lemma 4.6]{CJKM}. The operators $T_j$ appearing in that proof are to be replaced with $T_{\varphi_k (r_{j-1} \: \cdot \: r_{j}^{-1} t_j)}$ for $1\leq j \leq n-1$ and $T_{n}$ is $T_{\varphi_k(t_n r_{n-1} \: \cdot \: r_n^{-1})}$. Since all $x_j$ are $1$, the term $S_{j+1}(x_{j+1} S_{j+2}(x_{j+2} \ldots S_{n-1}(x_{n-1}) \ldots ))$ in the definition of $R_{j,V}$ that appears in  \cite[Lemma 4.6]{CJKM} is now just the scalar
 \[
 \prod_{i=j+1}^{n-1} \varphi_k (r_{i-1} r_{i}^{-1} t_i).
 \]
 Now, in equation (4.13) of \cite{CJKM}, the commutator terms vanish, as one of the terms in every case is a scalar. Additionally, since the $\widetilde{S_j}$ in the estimate for the first summand in equation (4.14) of \cite{CJKM} is a scalar, we can once more avoid \cite[Lemma 3.15]{CJKM} and the SAIN condition.   Once again, note that since $x_j=1$, the proof remains valid even when some of the $p_j=\infty$.
\end{proof}

\begin{remark}
Theorem \ref{thm:transference} assumes $G$ to be second countable since its proof relies on \cite{CJKM} which assumes second countability.
\end{remark}

\begin{remark}
Fix $1 \leq i \leq n$.  In case $p_i = p  = 1$ and $p_j = \infty$ for all $1 \leq j \leq n, i \not = j$ we do not know whether Theorem \ref{thm:transference} holds. The reason is that we do not know whether the limits \eqref{Eqn=Limits} (at index $i$) hold and neither do we know if the two further applications of Proposition 3.9 in \cite[Lemma 4.6, proof]{CJKM} hold.
\end{remark}

\section{Amenable groups: transference from Schur to Fourier multipliers} \label{Sect=Schur->Fourier}

Recall \cite[Section G]{BekkaHarpeValette} that   $G$ is amenable iff it satisfies the F\o lner condition: for any $\epsilon>0$ and any compact set $K\subseteq G$, there exists a compact set $F$ with non-zero measure such that $\frac{\mu_G(s.F \Delta F)}{\mu_G(F)} < \epsilon$ for all $s\in K$. Here $\Delta$ is the symmetric difference. This allows us to construct a net $F_{(\epsilon,K)}$ of such F\o lner sets using the ordering $( \epsilon_1, K_1) \leq (\epsilon_2,K_2)$ if $\epsilon_1 \geq \epsilon_2, K_1 \subseteq K_2$.

\begin{theorem}\label{thm:amenable-intertwining}
Let $G$ be a locally compact, unimodular, amenable group and let $1\leq p,p_1,\ldots,p_n \leq \infty$ be such that $p^{-1}=\sum_{i=1}^n p_i^{-1}$. Let $\phi \in C_b(G^{\times n})$ be such that $\widetilde{\phi}$ is the symbol of a $(p_1,\ldots,p_n)$-Schur multiplier of $G$. Then there is an ultrafilter $\cU$ on a set $I$ and there are complete contractions (resp. complete isometries if $1 < q < \infty$)   $i_q: L_q(\cL G) \to \prod_{\cU} S_q(L_2(G))$ such that, for all $f_i,f \in C_c(G)\star C_c(G)$,
\begin{equation}\label{Eqn=AmenableIntertwining}
\langle i_p(T_\phi(x_1,\ldots,x_n)),i_{p'}(y^*) \rangle_{p,p'} = \left \langle \left( M_{\widetilde{\phi}}( i_{p_1,\alpha}(x_1),\ldots,i_{p_n,\alpha}(x_n) ) \right)_{\alpha \in I} , i_{p'}(y^*) \right \rangle_{p,p'},
\end{equation}
where $x_i = \lambda(f_i), y = \lambda(f)$, and $\frac{1}{p'}+\frac{1}{p}=1$. In a similar way, the matrix amplifications of $i_q$ intertwine the multiplicative amplifications of the Fourier and Schur multipliers.
\end{theorem}
\begin{proof}
Let $F_\alpha, \alpha \in I$ be a F\o lner net for $G$, where $I$ is the index set consisting of pairs $(\epsilon, K)$ for $\epsilon > 0$, $K \subseteq G$ compact. It has the ordering as described above. Let $P_\alpha=P_{F_\alpha}$ be the projection onto $L_2(F_\alpha)$. Let $\cU$ be an ultrafilter refining the net $I$, and consider the map $i_p: L_p(\cL G)\to \prod_\cU S_p(L_2(G))$ defined by $i_p(x) = (i_{p,\alpha}(x))_{\alpha \in I} = (\frac{1}{\mu_G(F_\alpha)^{1/p}}P_{\alpha}x P_{\alpha})_{\alpha \in I}$. From \cite[Theorem 5.1]{CaspersDeLaSalle}, $i_p$ is a complete contraction (and even a complete isometry for $1< p < \infty$ \cite[Theorem 5.2]{CaspersDeLaSalle}); here the F\o lner condition is used.

Fix $\alpha$ and let $f \in C_c(G)\ast C_c(G)$. We first observe that by \cite[Theorem 5.1]{CaspersDeLaSalle} applied to the bounded operator $x = \lambda(f)$ we have $P_{F_{\alpha}} \lambda(f) P_{F_\alpha} \in S_q(L_2(G))$ for all $1\leq q \leq \infty$ and the kernel of this operator is given by the function

\[
(s,t) \mapsto 1_{F_\alpha}(s) f(st^{-1}) 1_{F_\alpha}(t).
\]
So we have
\begin{align*}
M_{\widetilde{\phi}}  (i_{p_1, \alpha}(x_1),\ldots,i_{p_n, \alpha}(x_n)) (t_0,t_{n}) & \\ =
\frac{1}{\mu_G(F_\alpha)^{1/p}} 1_{F_\alpha}(t_0)1_{F_\alpha}(t_n) \int_{F_\alpha^{\times n-1}} & \phi(t_0t_1^{-1},\ldots, t_{n-1}t_n^{-1}) f_1(t_0t_1^{-1})\ldots f_n(t_{n-1}t_n^{-1})  dt_1 \ldots dt_{n-1}.
\end{align*}
Moreover, after some change of variables, we see that $P_{F_\alpha} T_\phi(x_1,\ldots,x_n)P_{F_\alpha}$ is given by the kernel
\begin{align*}
(t_0,t_n) \mapsto 1_{F_\alpha}(t_0)  \int_{G^{\times n-1}}  \phi(t_0t_1^{-1},\ldots,t_{n-1}t_n^{-1}) &f_1(t_0t_1^{-1})\ldots f_n(t_{n-1}t_n^{-1})  1_F(t_n) dt_1\ldots dt_{n-1}.
\end{align*}
Let $\Phi$ denote the function
\[
 \Phi(t_0,\ldots,t_n) = \phi(t_0t_1^{-1},\ldots,t_{n-1}t_n^{-1}) f_1(t_0t_1^{-1})\ldots f_n(t_{n-1}t_n^{-1}) f(t_n t_0^{-1}),
\]
 and let $\Psi_\alpha$ be defined by
\[
\Psi_\alpha(t_0,\ldots, t_n) = 1_{F_\alpha}(t_0) 1_{F_\alpha} (t_n) - 1_{F_\alpha^{\times n+1}}(t_0,\ldots, t_n).
\]

Let $K$ be some compact set such that $\supp(f_j),\supp(f)\subseteq K$.  Let $t_0, \dots, t_n$ be such that both $\Phi(t_0, \dots, t_n)$ and $\Psi_\alpha(t_0, \dots, t_n)$ are nonzero. Since $\Psi_\alpha(t_0, \dots, t_n)$ is nonzero, we must have $t_0, t_n \in F_\alpha$ and $t_1, \ldots ,t_{n-1} \notin F_\alpha$. Since $\Phi(t_0, \dots, t_n)$ is nonzero, there are $k_1,\ldots,k_n \in K$ such that $t_{n-1} = k_n t_n,\ t_{n-2} = k_{n-1}k_n t_n,\ \ldots  ,\ t_0 = k_1 \ldots k_n t_n$. Hence we find that $t_n$ belongs to the set
\[
F_\alpha \cap F_\alpha.(k_1\ldots k_n)^{-1} \setminus \left( F_\alpha.(k_2\ldots k_n)^{-1} \cup \ldots \cup F_\alpha.k_n^{-1} \right).
\]

Using these facts, along with some change of variables, we get
\begin{equation} \label{Eqn=UltraIntertwining}
\begin{split}
&\vert \langle i_{p,\alpha}(T_\phi(x_1,\ldots,x_n)),i_{p',\alpha}(y^*) \rangle_{p,p'} - \langle M_{\widetilde{\phi}}( i_{p_1,\alpha}(x_1),\ldots,i_{p_n,\alpha}(x_n) )  , i_{p',\alpha}(y^*) \rangle_{p,p'} \vert \\
&= \left\vert \frac{1}{\mu_G(F_\alpha)} \int_{G^{n+1}} \Phi(t_0,\ldots,t_n) \Psi_\alpha(t_0,\ldots,t_n) dt_0 \ldots dt_n \right\vert \\
&= \left\vert \frac{1}{\mu_G(F_\alpha)} \int_{K^{n}} \int_{G} \Phi(k_1\ldots k_n t_n, \ldots , k_n t_n, t_n)  \Psi_\alpha(k_1\ldots k_nt_n,\ldots ,k_nt_n, t_n)dk_1 \ldots dk_n dt_n \right\vert \\
&\leq \| \Phi \|_\infty \int_{K^n}  \frac{1}{\mu_G(F_\alpha)}\mu_G\left(F_\alpha \cap (F_\alpha.(k_1\ldots k_n)^{-1}) \cap (F_\alpha.(k_2\ldots k_n)^{-1})^c \cap \ldots \cap (F_\alpha.(k_n)^{-1})^c\right) dk_1 \ldots dk_n \\
& \leq \| \Phi \|_\infty \int_{K^n} \frac{1}{\mu_G(F_\alpha)}\mu_G\left(F_\alpha \cap(F_\alpha.(k_n)^{-1})^c \right) dk_1 \ldots dk_n.
\end{split}
\end{equation}
Using the ordering described earlier, if the index $\alpha \geq (\epsilon \times \left(\|\Phi\|_\infty \mu_G(K^n)\right)^{-1},  K^{-1})$, then the F\o lner condition implies that \eqref{Eqn=UltraIntertwining} is less than $\epsilon$, and hence equation \eqref{Eqn=AmenableIntertwining} is true.

A direct modification of this argument now shows that the $i_p$ also intertwine the multiplicative amplifications of the Fourier and Schur multipliers. i.e. for $\beta_i \in S_{p_i}^N, \beta \in S_{p'}^{N}$, we have

\begin{equation}
\begin{split}
 \langle \id \otimes & i_p(T^{(N)}_\phi(\beta_1\otimes x_1,\ldots,\beta_n \otimes x_n)),  \id\otimes i_{p'}(\beta \otimes y^*) \rangle_{p,p'} = \\
   &\left\langle \left( M^{(N)}_{\widetilde{\phi}}(\id\otimes i_{p_1,\alpha}(\beta_1 \otimes x_1),\ldots,\id\otimes i_{p_n,\alpha}(\beta_n \otimes x_n) ) \right)_\alpha , \id\otimes i_{p'}(\beta\otimes y^*) \right\rangle
  \end{split}
  \end{equation}
\end{proof}

Combining this with Theorem \ref{thm:transference}, we get the multilinear version of \cite[Corollary 5.3]{CaspersDeLaSalle}

\begin{corollary}\label{Cor = AmenableTransference}
Let $1 <  p,p_1,\ldots,p_n < \infty$ be such that $p^{-1}=\sum_{i=1}^n p_i^{-1}$. Let $\phi \in C_b(G^{\times n})$ and assume that $G$ is amenable. If $\widetilde{\phi}$ is the symbol of a $(p_1, \ldots, p_n)$-bounded (resp. multiplicatively bounded) Schur multiplier then $\phi$ is the symbol of a $(p_1, \dots, p_n)$-bounded (resp. multiplicatively bounded) Fourier multiplier.
 Moreover,
\[
\Vert T_\phi \Vert_{(p_1,\ldots,p_n)} \leq \Vert M_{\widetilde{\phi}} \Vert_{(p_1,\ldots,p_n)}, \qquad
\Vert T_\phi \Vert_{(p_1,\ldots,p_n)-mb} \leq \Vert M_{\widetilde{\phi}} \Vert_{(p_1,\ldots,p_n)-mb},
\]
with equality in the $(p_1,\ldots,p_n)$-$mb$ norm when $G$ is second countable.
\end{corollary}

\begin{proof}
 For $1<p<\infty$, $i_p$ is a (complete) isometry. Hence, for $x_i$ as in the hypothesis of Theorem \ref{thm:amenable-intertwining}, we get
\[
\begin{split}
\Vert T_\phi (x_1,\ldots,x_n) \Vert_{L_p(\cL G)} = & \Vert i_p \circ T_\phi (x_1,\ldots,x_n) \Vert_{L_p(\cL G)} = \Vert M_{\widetilde{\phi}} (i_{p_1}(x_1),\ldots,i_{p_n}(x_n)) \Vert_{L_p(\cL G)} \\
 \leq & \Vert M_{\widetilde{\phi}} \Vert_{(p_1,\ldots p_n)} \prod_{i=1}^n \Vert x_i \Vert_{L_{p_i}(\cL G)}.
\end{split}
\]
For $1<p_i<\infty$, such $x_i$ are norm dense in $L_{p_i}(\cL G)$, so we get the bound on $\Vert T_\phi \Vert_{(p_1,\ldots,p_n)}$. The multiplicatively bounded version follows similarly, with the other inequality coming from Theorem \ref{thm:transference}.
\end{proof}

Let $H\leq G$ be a subgroup. Clearly, from Theorem \ref{Thm=Finite truncation}, the restriction of $\phi$ to $H$ also determines a bounded Schur multiplier, with
\[ \| M_{\widetilde{\phi_{|H}}} \|_{(p_1,\ldots,p_n)} \leq \| M_{\widetilde{\phi}}\|_{(p_1,\ldots,p_n)}. \]
Combining this observation with Corollary \ref{Cor = AmenableTransference} gives us the multiplicatively bounded version of the multilinear de Leeuw restriction theorem \cite[Theorem 4.5]{CJKM} for amenable discrete subgroups of second countable groups.  Note that the SAIN condition used in \cite[Theorem 4.5]{CJKM} is implicit here since the subgroup is amenable. 

\begin{corollary}
Let $G$ be a locally compact, unimodular, second countable group and let $1<p,p_1,\ldots,p_n<\infty$ with $p^{-1}=\sum_i p_i^{-1}$. Let $\phi \in C_b(G^{\times n})$ be a symbol of a $(p_1,\ldots,p_n)$-multiplicatively bounded Fourier multiplier.  If $H$ is an  amenable discrete subgroup of $G$, then we have
\[
\| T_{\phi|_H} \|_{(p_1,\ldots,p_n)-mb} \leq \|T_\phi \|_{(p_1,\ldots, p_n)-mb}.
\]

\end{corollary}

\begin{remark} \label{Rmk=NonUnimodular}
We now discuss some difficulties one encounters when modifying the above methods to the non-unimodular case, meaning that the Plancherel weight $\varphi$ is no longer a trace. In this case $\l(C_c(G) \star C_c(G))$ is no longer a common dense subset of the $L_p(\cL G)$. Rather, we have embeddings $\eta_{t,p}: \l(C_c(G) \star C_c(G)) \to L_p(\cL G)$ with dense image given by $\l(f) \mapsto \Delta^{(1-t)/p} \l(f) \Delta^{t/p}$, where $0 \leq t \leq 1$ and $\Delta$ is the multiplication operator with the modular function, which we denote also by $\Delta$ (see \cite{TerpThesis} and \cite{TerpInterpolation} for details). \\

This raises the question how to define the $(p_1, \dots, p_n)$-Fourier multiplier. A possible choice would be to take some $t \in [0,1]$ and define
\[
    T_\phi(\eta_{t,p_1}(x_1), \ldots, \eta_{t, p_n}(x_n)) = \eta_{t,p}(T_\phi(x_1, \ldots, x_n)).
\]
However, with this definition, the intertwining property in the ultralimit of \eqref{Eqn=UltraIntertwining} will no longer hold. This can be illustrated by considering the case $n=2$ and $\phi(x,y) = \phi_1(x) \phi_2(y)$. In order for the intertwining property to hold, the Fourier multiplier would have to satisfy
\[
    T_\phi(\eta_{t,p_1}(x_1), \eta_{t,p_2}(x_2)) = T_{\phi_1}(\eta_{t,p_1}(x_1)) T_{\phi_2}(\eta_{t,p_2}(x_2))
\]
where on the right hand side we use the linear definition of the Fourier multiplier from \cite{CaspersDeLaSalle}. This is not the case with the above definition.
 As a consequence, we no longer have nice relations between nested linear Fourier multipliers and multilinear Fourier multipliers, as in \cite[Lemma 4.4]{CJKM}. As the proof of the multilinear restriction theorem in \cite{CJKM} repeatedly uses such formulae, it is also unclear if these de Leeuw type theorems are still valid in the non-unimodular case.
\end{remark}

\section{Domain of the completely bounded bilinear Hilbert transform and Calder\'on-Zygmund operators} \label{Sect=HilbertTransform}

In this final section we prove a result about non-boundedness of the bilinear Hilbert transform based on our multilinear transference techniques. We prove an analogous result for examples of Calder\'on-Zygmund operators. This shows that the main results from \cite{Amenta}, \cite{DiPlinio22} about $L_p$-boundedness of certain Fourier multipliers cannot be extended to range spaces with $p \leq 1$. This is in contrast with the Euclidean (non-vector valued) case.

\subsection{Lower bounds for the vector valued bilinear Hilbert transform}
For $0 < p < \infty$ let $S_p^N = S_p(\mathbb{C}^N)$ be the Schatten $L_p$-space associated with linear operators on $\mathbb{C}^N$. For $ 0 < p < 1$ we have that $S_p^N$ is a quasi-Banach space satisfying the quasi-triangle inequality:
\[
\Vert x + y \Vert_p \leq 2^{\frac{1}{p}-1} (\Vert x \Vert_p + \Vert y \Vert_p), \qquad x,y \in S_p^N.
\]
 We set
  \[
  h(\xi_1, \xi_2) =   \chi_{\geq 0}(\xi_1 - \xi_2), \qquad \xi_1, \xi_2 \in \mathbb{R}.
  \]
The first statement of the following theorem is the main result of \cite{LaceyThiele} and the latter statement of this theorem for $1 <  p < \infty$ was proved in \cite{Amenta} and \cite{DiPlinio22}.

\begin{theorem}\label{Thm=LaceyThiele}

For every $1 < q_1,q_2,q, p_1, p_2 < \infty, \frac{2}{3} < p < \infty$ with $\frac{1}{p} = \frac{1}{p_1} + \frac{1}{p_2}$ and $N \in \mathbb{N}_{\geq 1}$ there exists a bounded linear map
\begin{equation}\label{Eqn=HilbertTransform}
T_h^{(N)}: L_{p_1}(\mathbb{R}, S_{q_1}^N ) \times  L_{p_2}(\mathbb{R}, S_{q_2}^N ) \rightarrow L_{p}(\mathbb{R}, S_{q}^N )
\end{equation}
that is determined by
\[
T_{h}^{(N)}(f_1, f_2)(s) = \int_\mathbb{R}  \int_\mathbb{R}  \widehat{f_1}(\xi_1) \widehat{f_2}(\xi_2) h(\xi_1, \xi_2) e^{i s (\xi_1 + \xi_2)} d\xi_1 d\xi_2,
\]
where $s \in \mathbb{R}$ and $f_i, i =1,2$ are functions in $L_{p_i}(\mathbb{R}, S_{q_i}^N )$  whose Fourier transforms $\widehat{f_i}$ are continuous compactly supported functions $\mathbb{R} \rightarrow S_{q_i}^N$.  If $1 < p := (\frac{1}{p_1} + \frac{1}{p_2})^{-1} < \infty$ and $\frac{1}{\max\{q,q'\}}+\frac{1}{\max\{q_1,q_1'\}}+\frac{1}{\max\{q_2,q_2'\}} > 1$ we have that this operator is moreover uniformly bounded in $N$.\end{theorem}

  Note that the map $T_h^{(N)}$ as defined above coincides with the multiplicative amplification of the map $T_h := T_h^{(1)}$ as defined in Section \ref{Sect=MultilinearNorms}, so this notation is consistent.

Our aim is to show that the results of \cite{Amenta} and \cite{DiPlinio22} cannot be extended to the case { $p_i=q_i$, $q=p = 1=\frac{1}{p_1}+\frac{1}{p_2}$}; i.e. the bound of \eqref{Eqn=HilbertTransform} is not uniform in $N$. In particular we show that the bound can be estimated from below by $C \log(N)$ for some   constant $C$ independent of $N$. 

For a function $\phi: \mathbb{R}^2 \rightarrow \mathbb{C}$ we recall the definition
\[
\widetilde{\phi}(\lambda_0, \lambda_1, \lambda_2) = \phi(\lambda_0 - \lambda_1, \lambda_1 - \lambda_2), \qquad \lambda_i \in \mathbb{R}.
\]

\begin{theorem}\label{Thm=LowerBoundThiele}
Let $1  < p_1, p_2 < \infty$ be such that $\frac{1}{p_1} + \frac{1}{p_2} = 1$.
There exists an absolute constant $C >0$ such that for every $N \in \mathbb{N}_{\geq 1}$ we have
\[
A_{p_1, p_2,N} := \Vert T_h^{(N)}: L_{p_1}(\mathbb{R}, S_{p_1}^N ) \times  L_{p_2}(\mathbb{R}, S_{p_2}^N ) \rightarrow L_{1}(\mathbb{R}, S_{1}^N ) \Vert > C \log(N).
\]
\end{theorem}
\begin{proof}
In the proof let $\mathbb{Z}_N = [-N, N] \cap \mathbb{Z}$. We may naturally identify $S_p^{2N+1}$ with  $S_p(   \ell_2(\mathbb{Z}_N) )$.
 Let $\varphi \in C_c(\mathbb{R}), \varphi \geq 0$ be such that $\varphi(t) = \varphi(-t), t\in \mathbb{R}$,  $\Vert \varphi \Vert_{L_1(G)} = 1$ and its support is contained in $[-\frac{1}{2}, \frac{1}{2}]$. Set for $s_1, s_2 \in \mathbb{R}$,
\[
H(s_1, s_2) = \int_{\mathbb{R}} h(s_1 + t, -t + s_2 ) \varphi(t) dt.
\]
Then $H$ is continuous and $H$ equals $h$ on $\mathbb{Z} \times \mathbb{Z}$.
As a consequence of \cite[Lemma 4.3]{CJKM} we find
\[
\Vert T_H^{(2N+1)}: L_{p_1}(\mathbb{R}, S_{p_1}^{2N+1} ) \times  L_{p_2}(\mathbb{R}, S_{p_2}^{2N+1} ) \rightarrow L_{1}(\mathbb{R}, S_{1}^{2N+1} ) \Vert \leq A_{p_1, p_2,2N+1}.
\]
By the multilinear De Leeuw restriction theorem \cite[Theorem C]{CJKM} we have
\begin{equation}\label{Eqn=MultiplierEstimate}
\Vert T_{H\vert_{\mathbb{Z} \times \mathbb{Z}}}^{(2N+1)}: L_{p_1}(\mathbb{T}, S_{p_1}^{2N+1} ) \times  L_{p_2}(\mathbb{T}, S_{p_2}^{2N+1} ) \rightarrow L_{1}(\mathbb{T}, S_{1}^{2N+1} ) \Vert \leq  A_{p_1, p_2,2N+1}.
\end{equation}
Let $\zeta_l(z) = z^l, z\in \mathbb{T}, l \in \mathbb{Z}$. Set the unitary $U = \sum_{l=-N}^N p_l \otimes \zeta_l$ and for any $1 < p < \infty$ the isometric map
\[
\pi_p: S_p^{2N+1} \rightarrow S_p^{2N+1} \otimes L_p(\mathbb{T}): x \mapsto  U (x \otimes 1)U^\ast.
\]
Then,
\[
T_{H\vert_{\mathbb{Z} \times \mathbb{Z}}}^{(2N+1)} \circ (\pi_{p_1} \times \pi_{p_2}) = \pi_p \circ  M_{\widetilde{ H}\vert_{\mathbb{Z}_N \times \mathbb{Z}_N}}.
\]
 This together with \eqref{Eqn=MultiplierEstimate} implies that
\begin{equation}\label{Eqn=BilHil1}
\Vert  M_{\widetilde{ H}\vert_{\mathbb{Z}_N \times \mathbb{Z}_N}}:   S_{p_1}^{2N+1}  \times   S_{p_2}^{2N+1}  \rightarrow  S_{1}^{2N+1}   \Vert
\leq  A_{p_1, p_2,2N+1}.
\end{equation}
Now set $H_j(s,t) = \widetilde{ H } \vert_{\mathbb{Z} \times \mathbb{Z}}(s, j, t), s,t \in \mathbb{Z}$. Note that
\[
H_j(s,t) = { \chi_{\geq 0}(s+t-2j)}.
\]
 By \cite[Theorem 2.3]{PSSTAdvances} we find that
\begin{equation}\label{Eqn=BilHil2}
\max_{-N \leq j \leq N} \Vert M_{ H_j  }: S_{1}^{2N+1}  \rightarrow  S_{1}^{2N+1}  \Vert
\leq \Vert M_{\widetilde{ H }\vert_{\mathbb{Z}_N \times \mathbb{Z}_N}}^{(N)}:  S_{p_1}^{2N+1}  \times    S_{p_2}^{2N+1}  \rightarrow   S_{1}^{2N+1}  \Vert.
\end{equation}
For $j = 0$ we have that $M_{ H_j  }$ is the triangular truncation map and therefore by \cite[Proof of Lemma 10]{Davies} (apply $M_{H_0}$ to the matrix consisting of only 1's) there is a constant $C>0$ such that
\begin{equation}\label{Eqn=BilHil3}
C \log(2N+1) \leq    \Vert M_{ H_0  }: S_{1}^{2N+1}  \rightarrow  S_{1}^{2N+1}   \Vert.
\end{equation}
Combining \eqref{Eqn=BilHil1}, \eqref{Eqn=BilHil2}, \eqref{Eqn=BilHil3} yields the result for $2N+1$. Since the norm of $T_h^{(N)}$ is increasing in $N$ the result for even $N$ also follows.

\end{proof}

\subsection{Lower bounds for Calder\'on-Zygmund operators}

The aim of this section is to show a result similar to Theorem \ref{Thm=LowerBoundThiele} for Calder\'on-Zygmund operators by considering an example. This shows that the results from \cite{DiPlinio20} cannot be extended to the case where the range space is $p=1$. This is in contrast with the commutative situation where Grafakos-Torres \cite{GrafakosTorres} have shown boundedness of a   class of Calder\'on-Zygmund operators with natural size and smoothness conditions as maps $L_p \times \ldots \times L_p \rightarrow L_{p/n}$ for $p \in (1, \infty)$.

Consider any   symbol $m$ that is smooth on $\mathbb{R}^2 \backslash \{ 0 \}$, homogeneous and which is determined on one of the quadrants by
\begin{equation} \label{Eqn=Symbolm}
m(s,t) = \frac{s}{s-t}, \qquad s \in \mathbb{R}_{>0}, t \in \mathbb{R}_{<0}.
\end{equation}
Here homogeneous means that  $m(\lambda s, \lambda t) = m(s,t), s,t \in \mathbb{R}, \lambda >0$. We assume moreover that $m$ is regulated at 0, by which we mean that
\[
m(0) = \pi^{-1} r^{-2}   \int_{ \Vert (t_1, t_2) \Vert_2 \leq r} m(t_1, t_2) dt_1 d t_2, \qquad r > 0.
\]
As $m$ is homogeneous this expression is independent of $r$. This type of symbol $m$ is important as it occurs naturally  in the analysis of divided difference functions; for instance it plays a crucial role  in \cite{CaspersIsrael}.

\begin{theorem}\label{Thm=CZLowerBound}
Let $1  < p_1, p_2 < \infty$ be such that $\frac{1}{p_1} + \frac{1}{p_2} = 1$.
There exists an absolute constant $C >0$ such that
\[
B_{p_1, p_2,N} := \Vert T_m^{(N)}: L_{p_1}(\mathbb{R}, S_{p_1}^N ) \times  L_{p_2}(\mathbb{R}, S_{p_2}^N ) \rightarrow L_{1}(\mathbb{R}, S_{1}^N ) \Vert > C \log(N).
\]
\end{theorem}
\begin{proof}
By \cite[Lemma 10]{Davies} (and the proof of \cite[Corollary 11]{Davies}) there exist constants $0 = \lambda_0 < \lambda_1 <  \ldots < \lambda_N$ such that the function
\[
\phi(i,j) = \frac{\lambda_i - \lambda_j}{\lambda_i + \lambda_j}, \qquad 1 \leq i,j \leq N,
\]
is the symbol of a linear Schur multiplier $M_\phi: S_1^{N} \rightarrow S_1^N$ whose norm is at least $C  \log(N)$ for some absolute constant $C > 0$. Without loss of generality we may assume that $\lambda_i \in K_N^{-1} \mathbb{Z}$ for some $K_N \in \mathbb{N}_{\geq 1}$ by an approximation argument.
Then in this proof let $\Lambda_N = \{ \lambda_0,  \lambda_1, \ldots,  \lambda_N \}$. We may naturally identify $S_p^{N+1}$ with  $S_p(  \ell_2(\Lambda_N))$ by identifying $E_{i,j}$ with $E_{\lambda_i, \lambda_j}$. We proceed as in the proof of Theorem \ref{Thm=LowerBoundThiele}.

For $\lambda \in K_N^{-1} \mathbb{Z}$ let $p_\lambda$ be the orthogonal projection of $\ell_2(K_N^{-1} \mathbb{Z})$ onto $\mathbb{C} \delta_\lambda$.
Further for $\lambda \in K_N^{-1} \mathbb{Z}$ set $\zeta_\lambda: \mathbb{T} \rightarrow \mathbb{C}$ by  $\zeta_\lambda(z) = z^{  K_N \lambda}, \theta \in \mathbb{R}$. This way every $z \in \mathbb{T}$ determines a representation $\lambda \mapsto \zeta_\lambda(z)$ of $K_N^{-1} \mathbb{Z}$ and this identifies $\mathbb{T}$ with the Pontrjagin dual of $K_N^{-1} \mathbb{Z}$.
Set the unitary $U = \sum_{\lambda \in \Lambda_N}  p_\lambda \otimes \zeta_\lambda$ and for any $1 < p < \infty$ the isometric map
\[
\pi_p: S_p^{N+1} \rightarrow S_p^{N+1} \otimes L_p(\mathbb{T}): x \mapsto  U (x \otimes 1)U^\ast.
\]

  For $r>0$, consider the function
\[
m_r (s_1,s_2) = \frac{1}{\pi r^2} \int_{\Vert (s_1-t_1,s_2-t_2) \Vert_2 \leq r } m(t_1,t_2) dt_1 dt_2.
\]
This function is continuous and bounded, and hence we may apply the bilinear De Leeuw restriction theorem \cite[Theorem C]{CJKM} to get
\begin{equation}\label{Eqn=Leeuw}
\begin{split}
& \Vert T_{  m_r\vert_{ (K_N^{-1} \mathbb{Z})^2} }^{(N+1)}: L_{p_1}(\mathbb{T}, S_{p_1}^{N+1}) \times  L_{p_2}(\mathbb{T}, S_{p_2}^{N+1} ) \rightarrow L_{1}(\mathbb{T}, S_{1}^{N+1})   \Vert\\
&\qquad  \leq \Vert T_{ m_r }^{(N+1)}: L_{p_1}(\mathbb{R}, S_{p_1}^{N+1}) \times  L_{p_2}(\mathbb{R}, S_{p_2}^{N+1}) \rightarrow L_{1}(\mathbb{R}, S_{1}^{N+1} )   \Vert.
\end{split}
\end{equation}
 Since $ m_r\vert _{ (K_N^{-1} \mathbb{Z})^2}$ converges to $m \vert_{ (K_N^{-1} \mathbb{Z})^2}$ pointwise, we obtain (by considering the action of the multiplier on functions with finite frequency support),
\begin{equation}\label{Eqn=Approx}
\begin{split}
& \lim_{r\searrow 0} \Vert T_{ m_r\vert _{ (K_N^{-1} \mathbb{Z})^2 } }^{(N+1)}: L_{p_1}(\mathbb{T}, S_{p_1}^{N+1}) \times  L_{p_2}(\mathbb{T}, S_{p_2}^{N+1} ) \rightarrow L_{1}(\mathbb{T}, S_{1}^{N+1})  \Vert
\\
& \qquad  = \Vert T_{ m\vert_{ (K_N^{-1} \mathbb{Z})^2} }^{(N+1)}: L_{p_1}(\mathbb{T}, S_{p_1}^{N+1}) \times  L_{p_2}(\mathbb{T}, S_{p_2}^{N+1} ) \rightarrow L_{1}(\mathbb{T}, S_{1}^{N+1})\Vert.
\end{split}
\end{equation}
Further, viewing $m_r$ as a convolution of $m$ with an $L_1(\mathbb{R}^2)$ function, from \cite[Lemma 4.3]{CJKM},
\begin{equation}\label{Eqn=Convolve}
\begin{split}
&\Vert T_{ m_r }^{(N+1)}: L_{p_1}(\mathbb{R}, S_{p_1}^{N+1}) \times  L_{p_2}(\mathbb{R}, S_{p_2}^{N+1} ) \rightarrow L_{1}(\mathbb{R}, S_{1}^{N+1}) \Vert \\
&\qquad   \leq \Vert T_{ m}^{(N+1)}: L_{p_1}(\mathbb{R}, S_{p_1}^{N+1}) \times  L_{p_2}(\mathbb{R}, S_{p_2}^{N+1} ) \rightarrow L_{1}(\mathbb{R}, S_{1}^{N+1})   \Vert = B_{p_1, p_2,N}.
\end{split}
\end{equation}
Combining the estimates  \eqref{Eqn=Approx}, \eqref{Eqn=Leeuw}, \eqref{Eqn=Convolve} we find that
\begin{equation}\label{Eqn=CZmiddle}
\Vert T_{ m\vert_{ (K_N^{-1} \mathbb{Z})^2} }^{(N+1)}: L_{p_1}(\mathbb{T}, S_{p_1}^{N+1}) \times  L_{p_2}(\mathbb{T}, S_{p_2}^{N+1} ) \rightarrow L_{1}(\mathbb{T}, S_{1}^{N+1})\Vert\leq B_{p_1, p_2,N}.
\end{equation}

We view  $\widetilde{m}  \vert_{\Lambda_N \times \Lambda_N \times \Lambda_N}$ as the symbol of a Schur multiplier $M_{  \widetilde{ m}  \vert_{\Lambda_N \times \Lambda_N \times \Lambda_N}   }: S_{p_1}^{N+1}  \times   S_{p_2}^{N+1}  \rightarrow  S_{1}^{N+1}$.
  Then,
\[
T_{ m\vert_{K_N^{-1} \mathbb{Z}} }^{(N+1)} \circ (\pi_{p_1} \times \pi_{p_2}) = \pi_p \circ  M_{  \widetilde{ m} \vert_{\Lambda_N \times \Lambda_N \times \Lambda_N}    }.
\]
 It follows with \eqref{Eqn=CZmiddle} that
\begin{equation}\label{Eqn=CZ1}
\begin{split}
&  \Vert  M_{ \widetilde{m} \vert_{\Lambda_N \times \Lambda_N \times \Lambda_N}  }:   S_{p_1}^{N+1}  \times   S_{p_2}^{N+1}  \rightarrow  S_{1}^{N+1}  \Vert \\
\leq &
\Vert T_{ m\vert_{K_N^{-1} \mathbb{Z}}}^{(N+1)}: L_{p_1}(\mathbb{T}, S_{p_1}^{N+1} ) \times  L_{p_2}(\mathbb{T}, S_{p_2}^{N+1} ) \rightarrow L_{1}(\mathbb{T}, S_{p}^{N+1} )   \Vert
 \leq B_{p_1, p_2,N+1}.
\end{split}
\end{equation}
By \cite[Theorem 2.3]{PSSTAdvances} we find that
\begin{equation}\label{Eqn=CZ2}
 \Vert M_{ \widetilde{m} \vert_{\Lambda_N \times \Lambda_N \times \Lambda_N}( \: \cdot \:, 0 \: \cdot \: ) }: S_{1}^{N+1}  \rightarrow  S_{1}^{N+1}  \Vert
\leq \Vert M_{ \widetilde{m} \vert_{\Lambda_N \times \Lambda_N \times \Lambda_N}  }:  S_{p_1}^{N+1}  \times    S_{p_2}^{N+1}  \rightarrow   S_{p}^{N+1}  \Vert.
\end{equation}
Now for $s,t \in \mathbb{R}_{> 0}$ we find
\[
\widetilde{m}(s, 0, t) = m(s-0, 0-t) =  \frac{s}{s+t} = \frac{1}{2} ( 1 + \frac{s - t}{s+t}) = \frac{1}{2} ( 1 + \phi(s,t)).
\]
  It follows therefore by the first paragraph that for some constant $C>0$,
\[
  C \log(N) \leq
\Vert M_{ \widetilde{m} \vert_{\Lambda_N \times \Lambda_N \times \Lambda_N}( \: \cdot \:, 0 \: \cdot \: ) }: S_{1}^{N+1}  \rightarrow  S_{1}^{N+1}  \Vert.
\]
The combination of the latter estimate with  \eqref{Eqn=CZ1}   yields the result.
\end{proof}

\begin{remark}
  In \cite{GrafakosTorres} it is shown that for a natural class of Calder\'on-Zygmund operators, the associated convolution operator is bounded as a map $L_1 \times L_1 \rightarrow L_{\frac{1}{2}, \infty}$ as well as $L_{p_1} \times L_{p_2} \rightarrow L_p$ with $\frac{1}{p} = \frac{1}{p_1} + \frac{1}{p_2}$ and $\frac{1}{2} < p < \infty, 1<p_1,p_2<\infty$. This applies in particular to the map $T_m$ with symbol $m$ as in \eqref{Eqn=Symbolm}, see \cite[Proposition 6]{GrafakosTorres}. Our example shows that this result does not extend to the vector-valued setting in case $\frac{1}{2} < p \leq 1$. On the other hand, affirmative results in case  $1 < p < \infty$, and $\frac{1}{p}=\frac{1}{p_1}+\frac{1}{p_2}$ were obtained in \cite{DiPlinio20}. The question remains open whether a weak $L_1$-bound $L_{p_1} \times L_{p_2} \rightarrow L_{1, \infty}, \frac{1}{p_1} + \frac{1}{p_2} = 1$ holds, even in the case $p_1 = p_2 = 2$.
\end{remark}

\end{document}